\documentclass[11pt]{amsart}
\pdfoutput=1
\usepackage{graphicx}
\usepackage{tikz}
\usepackage{subfigure}
\newif\iftechexplorer\techexplorerfalse
\usepackage{amssymb,amsmath,amsthm,amsfonts,amscd}

\usetikzlibrary{matrix,arrows,decorations.markings}
\usepackage{verbatim}
	
\newtheorem{theorem}{Theorem}[section]
\newtheorem{Theorem}{Theorem}
\newtheorem{Theorem*}{Theorem}
\newtheorem{lemma}[theorem]{Lemma}
\newtheorem{proposition}[theorem]{Proposition}

\newtheorem{corollary}[theorem]{Corollary}

\newtheorem{definition}[theorem]{Definition}

\def\T{\mathcal{T}}
\def\K{\mathcal{K}}
\def\C{\mathcal{C}}
\def\N{\mathbb{N}}
\def\ZZ{\mathbb{Z} \oplus \mathbb{Z}}
\def\F{\mathcal{F}}
\def\S{\mathcal{S}}

\numberwithin{equation}{subsection}

\begin{document}
\title{Knot Floer Filtration Classes of Topologically Slice Knots}
\author{Joshua Tobin}
\date{\today}
\maketitle

\tikzset{->-/.style={decoration={
	markings,
	mark=at position #1 with
	{\arrow{>}}},postaction={decorate}}
	}

\begin{abstract}
The knot Floer complex and the concordance invariant $\varepsilon$ can be used to define a filtration on the smooth concordance group. We exhibit an ordered subset of this filtration that is isomorphic to $\mathbb{N} \times \mathbb{N}$ and consists of topologically slice knots.
\end{abstract}

\section{Introduction}

The \emph{(smooth) concordance group}, denoted $\C$, is the set of all knots in $S^3$ modulo concordance, where two knots $K_1$ and $K_2$ are \emph{concordant}, denoted $K_1 \sim_C K_2$, if  $K_1$ and $K_2$ cobound a smooth, properly embedded cylinder in $S^3 \times [0,1]$. The stipulation that the embedding of the cylinder must be smooth may be relaxed to require only a locally flat topological embedding, resulting in \emph{topological concordance} and the \emph{topological concordance group}. 

The identity element of the concordance group is the equivalence class of the unknot.  Any element of its smooth concordance class is called \emph{slice}, and any element of its topological concordance class is \emph{topologically slice}. The subgroup of $\C$ consisting of topologically slice knots, denoted $\C_{TS}$, is an important object of study because of its relevance to the existence of exotic $\mathbb{R}^4$s.  

In \cite{OS04}, Ozsv\'{a}th and Szab\'o (and indepentently Rasmussen in \cite{Ras03}) defined the \emph{knot Floer complex} associated with each knot. In \cite{Hom11}, Hom defined a group structure on a quotient of the set of knot Floer complexes. The primary utility of this group, denoted $\F$, is the existence of a homomorphism $\C \rightarrow \F$ and a filtration on $\F$ that can be pulled back to a filtration on $\C$ called the \emph{knot Floer filtration}.

In this paper, we explore the structure of $\C_{TS}$ by proving the following theorem.
\begin{Theorem}
\label{mainthm}
The indexing set of the knot Floer filtration of $\C_{TS}$ contains a subset that is order isomorphic to $\N \times \N$. We can index the filtration by 
\begin{equation*}
S = \{(i,j) \mid i \in \mathbb{N}, j \in \mathbb{N} \}
\end{equation*}
where S has the lexicographical ordering. Furthermore, each successive quotient is infinite, i.e. for $(i, j), (i',j') \in S$ and $(i, j) < (i',j')$, we have
\begin{equation*}
\mathbb{Z} \subset \F_{(i',j')}/ \F_{(i, j)}. 
\end{equation*}
\end{Theorem}

The structure of $\C_{TS}$ has been the subject of several previous results. In \cite{End95}, Endo showed that $\C_{TS}$ contains a subgroup isomorphic to $\mathbb{Z}^{\infty}$. Our result also extends the work of Hom in \cite{Hom11}, who showed that the indexing set of the knot Floer filtration of $\C$ contains a subset that is order isomorphic to $\mathbb{N}$ and has topologically slice representatives. It can also be compared to the work of Hancock, Hom, and Newman, who in \cite{HHN12} found a subset of the indexing set of the knot Floer filtration that is order isomorphic to $\mathbb{N} \times \mathbb{Z}$, but whose representatives are not topologically slice. 

The remainder of the paper is organized into two parts. Section 2 provides background on knot Floer homology and the group $\F$, including some necessary results about the ordering on $\F$. Section 3 will present the proof of Theorem \ref{mainthm}. The proof itself involves computing the knot Floer complexes of several families of knots arising as iterated cables of the Whitehead double of the trefoil knot. We will state the results of these computations first without proof and use them to demonstrate the proof of Theorem \ref{mainthm}. The computations will be performed in detail in the last section of the paper. 

\subsection*{Acknowledgements}
I would like to thank Michael Gartner, James Groccia, Christian Moscardi, Siran Li, Professor Jennifer Hom, and Professor Peter Horn for their involvement in the 2012 Columbia REU, where the ideas for this paper began. I am particularly grateful to Professor Hom, whose continual involvement and advice throughout the process of writing this paper has been invaluable. 

%-----------------------------------------------------------
%
%
%				knot Floer homology and F
%
%
%-----------------------------------------------------------

\section{Knot Floer homology and the group $\F$}
\label{kfh}

\subsection{Knot Floer homology}

Ozsv\'ath and Szab\'o \cite{OS04}, and independently Rasmussen \cite{Ras03}, associated to each knot $K \subset S^3$ a $\mathbb{Z}$-filtered chain complex over $\mathbb{F}[U,U^{-1}]$ ($\mathbb{F}$ denotes the field $\mathbb{Z}/2\mathbb{Z}$ and $U$ is a formal variable) known as the \emph{knot Floer complex} of $K$ and denoted $CFK^{\infty}(K)$. The homological $\mathbb{Z}$-grading on $CFK^{\infty}(K)$ is called the \emph{Maslov grading} $M$, and the $\mathbb{Z}$ filtration the \emph{Alexander filtration} $A$. The $-(U$-exponent$)$ induces a second $\mathbb{Z}$ filtration structure on $CFK^{\infty}$, allowing us to view it as a $\ZZ$-filtered chain complex with ordering given by $(i, j) \leq (i',j')$ if and only if $i \leq i'$ and $j \leq j'$. 

The filtered chain homotopy type of $CFK^{\infty}(K)$ is a knot invariant. The knot Floer complex of the connected sum of two knots $K_1$ and $K_2$ is the tensor product of the knot Floer complexes of $K_1$ and $K_2$, and the knot Floer complex of $-K$ is $CFK^{\infty}(K)^*$, the dual of the knot Floer complex of $K$. 

It is often helpful to think of $CFK^{\infty}(K)$ as a set of points and arrows in the lattice $\ZZ$. $CFK^{\infty}(K)$ is a free module over $\mathbb{F}[U, U^{-1}]$, so it is generated by some set $\{ x_i\}$. An element of the form $U^j x_i$ is placed at position $(-j,A(x_i)-j)$. Given a basis $\{ x_i \}$ for $CFK^{\infty}$, the differential can be depicted by arrows; one from each $x_i$ to each $x_j$ that appears with nonzero coefficient in $\partial x_i$. The length of the arrow from $x_i$ to $x_j$ represents the change in $\mathbb{Z} \oplus \mathbb{Z}$-filtration from $x_i$ to $x_j$. 

Given a subset $S \subset \ZZ$, let $C\{S\}$ be the subset of $CFK^{\infty}(K)$ consisting of elements of $CFK^{\infty}(K)$ that are in $S$ (along with the quotient differential). If for all $(i,j) \in S$, $(i',j') \leq (i,j)$ implies that $(i',j')$ is also in $S$, then $C\{S\}$ is a subcomplex of $C$. For appropriately defined $S$, $C\{S\}$ can also inherit a quotient or subquotient complex structure. For example, for any integer $s$, $C\{i=0, j \leq s\}$ naturally inherits a quotient subcomplex structure as the quotient of the subcomplex $C\{i \leq 0, j \leq s\}$ by the subcomplex $C \{ i \leq -1, j \leq s \}$.

\subsection{The invariant $\varepsilon$}
\label{Fdef}
$\F$ will be defined in this subsection as a set of $\varepsilon$-equivalence classes of knot Floer complexes, defined below. The definition of the invariant $\varepsilon$ due to Hom in \cite{Hom12} will rely on two other concordance invariants $\tau$ and $\nu$ both due to Ozsv\'{a}th and Szab\'{o}. The integer valued concordance invariant $\tau$ was defined in \cite{OS03} as
\begin{equation*}
\tau(K) = \text{min} \{s \mid \imath:C\{i=0, j \leq s\} \rightarrow C \{i=0 \} \text{ induces a non-trivial map on homology} \}
\end{equation*}
where $\imath$ denotes filtered chain complex inclusion.

Now to define $\nu$ as in \cite{OS11}, we first define subquotient complexes 
\begin{equation*}
A_s = C \{ \text{max}(i,j-s)=0 \} \text{ and }A_s' = C \{ \text{min}(i,j-s) = 0 \}
\end{equation*}
and maps 
\begin{equation*}
\nu_s:A_s \rightarrow C \{i = 0\} \text{ and } \nu_s': C\{i = 0 \} \rightarrow A_s' ,
\end{equation*}
where $\nu_s = \imath \circ \pi$  for $\pi: A_s \rightarrow C \{i = 0, j \leq s\}$ the quotient map and $\imath:C \{i = 0, j \leq s\} \rightarrow C \{ i = 0 \}$ the inclusion map, and $\nu_s' = \imath' \circ \pi'$ for $\pi': C\{i = 0 \} \rightarrow C \{ i = 0, j \geq s \}$ the quotient map and $\imath':C \{ i = 0, j \geq s \} \rightarrow A_s'$ the inclusion map. 

Define 
\begin{equation*}
\nu(K) = \text{min} \{ s \in \mathbb{Z} \mid v_s: A_s \rightarrow C\{i=0\} \text{ induces a non-trivial map on homology} \},
\end{equation*}
and similarly
\begin{equation*}
\nu'(K) = \text{max} \{ s \in \mathbb{Z} \mid v_s':C \{i = 0 \} \rightarrow A_s' \text{ induces a non-trivial map on homology} \}.
\end{equation*}
Then we are ready to define $\varepsilon$ as in \cite{Hom12}. 

\begin{definition}
The invariant $\varepsilon$ is defined in terms of $\tau$ and $\nu$ as
\begin{equation*}
\varepsilon (K) = 2 \tau(K) - \nu (K) - \nu'(K) .
\end{equation*}
\end{definition}

The invariant $\varepsilon$ is $\{-1,0,1\}$-valued because $\nu(K) = \tau(K)$ or $\tau(K) + 1$ and $\nu'(K) = \tau(K)$ or $\tau(K) -1$. It has the following properties, also from \cite{Hom12}. 
\begin{enumerate}
\item If $K$ is slice, then $\varepsilon(K) = 0$. 
\item $\varepsilon(-K) = -\varepsilon(K)$. 
\item If $\varepsilon(K) = \varepsilon(K')$, then $\varepsilon(K \# K') = \varepsilon(K)$. 
\item If $\varepsilon(K) = 0$, then $\varepsilon(K \# K') = \varepsilon(K')$. 
\end{enumerate}

If $C$ is a filtered chain complex, let $C_{i,j}$ denote the $(i, j)^{\text{th}}$-filtered subcomplex of $C$. A basis $\{ x_k \}$ for a filtered chain complex $C$ over $\mathbb{F}[U,U^{-1}]$ is a \emph{filtered basis} if $\{U^n \cdot x_k \mid U^n \cdot x_k \in C_{i,j}, n \in \mathbb{Z} \}$ is a basis for $C_{i,j}$ over $\mathbb{F}$ for all $i, j \in \mathbb{Z}$. Given a filtered basis $\{x_k \}$ for $C$, we can perform a filtered change of basis to produce a new filtered basis $\{x'_k \}$ for $C$, where
\begin{equation*}
x'_k = \begin{cases}
x_k + x_l & k = n \\
x_k & k \neq n
\end{cases}
\end{equation*}
for some $n$ and $l$ such that $x_l$ is in a filtration level that is less than or equal to the filtration level of $x_n$. 

The description of $\varepsilon$ becomes simpler given a filtered basis for $CFK^{\infty}(K)$. A basis $\{x_k \}$ of a $\ZZ$-filtered chain complex is called \emph{vertically simplified} if for each basis element, exactly one of the following holds:
\begin{itemize}
\item There is a unique incoming vertical arrow into $x_i$ and no outgoing vertical arrows.
\item There is a unique outgoing vertical arrow from $x_i$ and no outgoing vertical arrows.
\item There are no vertical arrows entering or leaving $x_i$.
\end{itemize}
A \emph{horizontally simplified basis} is defined similarly using horizontal arrows on basis elements. Every $\ZZ$-filtered chain complex has a vertically simplified basis and a horizontally simplified basis by \cite{LOT08} Proposition 11.52. It is not known whether every $\ZZ$-filtered chain complex arising as a knot Floer homology complex has a basis that is simultaneously vertically and horizontally simplified. However, it is possible to find simultaneously vertically and horizontally simplified bases for all of the $\ZZ$-filtered chain complexes discussed in this paper. 

For every knot $K$,there is a unique element of a vertically simplified basis for $CFK^{\infty}(K)$ with no incoming or outgoing arrows. Call this the \emph{vertically distinguished} element. By \cite{Hom12}, we can always find a horizontally simplified basis for $CFK^{\infty}(K)$ in which one of the basis elements $x_0$ is the vertically distinguished element of some vertically simplified basis. Given $x_0$, we can define $\varepsilon(K)$ in terms of such a basis as:
\begin{enumerate}

\item $\varepsilon (K) = 1$ if there is a unique incoming horizontal arrow into $x_0$.

\item $\varepsilon(K) = -1$ if there is a unique outgoing horizontal arrow from $x_0$.

\item $\varepsilon(K) = 0$ if there are no horizontal arrows entering or leaving $x_0$.

\end{enumerate}

\begin{definition}
Two knots $K_1$ and $K_2$ are $\varepsilon$-equivalent, denoted $\sim_{\varepsilon}$, if $\varepsilon(K_1 \# -K_2) = 0$.
\end{definition}

This is indeed an equivalence relation: reflexivity and symmetry are clear, and transitivity follows from the properties above since if $K \sim_{\varepsilon} J \sim_{\varepsilon} L$, then \begin{equation*}
\varepsilon(K \# -L) = \varepsilon(K \# -L \# (J \# -J)) = \varepsilon((K \# -J) \# (J \# -L)) = 0 .
\end{equation*}
Because $\varepsilon$-equivalence can also be seen as a property of the knot Floer complexes of two knots $K_1$ and $K_2$, we will sometimes write $CFK^{\infty}(K_1) \sim_{\varepsilon} CFK^{\infty}(K_2)$. With the definition of $\varepsilon$-equivalence, we can define the group $\F$ as follows.

\begin{definition}
The group $\F$ is 
\begin{equation*}
\F = (\{CFK^{\infty}(K) \mid K \subset S^3 \}, \otimes)/ \sim_\varepsilon .
\end{equation*}
\end{definition}
The first property of $\F$ to note is that there is a natural homomorphism $\C \rightarrow \F$ obtained by sending the equivalence class of a knot $K$ to the equivalence class of its knot Floer complex $CFK^{\infty}(K)$. We will explore some of the other properties of $\F$ in the next two subsections.

\subsection{L-space knots}
\label{lspacesub}

Because computation of the knot Floer complex of a knot is difficult in general, we will restrict our attention to $\varepsilon$-equivalence classes of a particular type of knot for which the knot Floer complex is known.

For any closed, oriented 3-manifold $Y$, \emph{Heegaard Floer homology}, first defined by Ozsv\'ath and Szab\'o in \cite{OS04}, associates to $Y$ a finitely generated abelian group $\widehat{HF}(Y)$. Also due to Ozsv\'ath and Szab\'o in \cite{OS05}, a space $Y$ is an \emph{L-space}  if $H_1(Y;\mathbb{Q})=0$ and $\widehat{HF}(Y)$ has rank equal to the number of elements in $H_1(Y;\mathbb{Z})$. All lens spaces are L-spaces, motivating the name.

For a knot $K$, if there exists a positive integer $n$ such that $n$ surgery on $K$ yields an L-space, then $K$ is called an \emph{L-space knot}. 

Some familiar families of knots are L-space knots. Let $K$ be a knot, and $(p,q)$ be a pair of relatively prime integers. Denote by $K_{p,q}$ the $(p,q)$-cable of $K$, where $p$ denotes the longitudinal winding. The iterated cable $(K_{p,q})_{r,s}$ will be denoted $K_{p,q;r,s}$. Hedden showed in \cite{Hed09} that if $K$ is an L-space knot and $q/p \geq 2 g(K) -1$, then $K_{p,q}$ is an L-space knot. Note that the unknot is an L-space knot, since $p/q$ surgery on the unknot gives $L(p,q)$. This implies that for $p$, $q > 0$, $T_{p,q}$ is an L-space knot, since the $(p,q)$ cable of the unknot is $T_{p,q}$.  This can also be seen directly by the fact that $pq \pm 1$ surgery on $T_{p,q}$ yields a lens space (see \cite{Mo}). 

Knowing that a knot $K$ is an L-space knot gives us two important pieces of information about $K$, as shown in \cite{OS05}.

The first is that there is an even positive integer $M$ and an increasing sequence $n_0, \hdots ,n_M$ of nonnegative integers with $n_i + n_{M-i} = 2g(K)$ and $n_0 = 0$ such that the Alexander polynomial of $K$ is of the form
\begin{equation*}
\label{lspace_form}
\Delta_K(t) = \displaystyle\sum\limits_{i=0}^{M} (-1)^i t^{n_i}.
\end{equation*}

The second is that the knot Floer complex of $K$ is determined by the sequence $n_i$. In fact, there is a filtered basis $\{x_i\}, i=0,\hdots ,M$ over $\mathbb{F}[U,U^{-1}]$ for $CFK^{\infty} (K)$ with differential given by 
\begin{equation*}
\partial x_i = \begin{cases}
x_{i-1} + x_{i+1} & i \text{ odd} \\
0 & i \text{ even}.
\end{cases}
\end{equation*}
Thinking of the knot Floer complex in the lattice, the arrow from $x_i$ to $x_{i-1}$ is horizontal of length $n_i - n_{i-1}$ and the arrow from $x_i$ to $x_{i+1}$ is vertical of length $n_{i+1} - n_i$. See Figure 1 for an example.

We will call any $\ZZ$-filtered chain complex over $\mathbb{F}[U,U^{-1}]$ that takes the above form a \emph{staircase complex} because of its appearance in the lattice $\ZZ$. These complexes will be denoted as
\begin{equation*}
(a_0, a_1, a_2, \hdots , a_n)
\end{equation*}
where for $i$ even, $a_i$ is the length of the horizontal arrow from $x_{i+1}$ to $x_i$, and for $i$ odd, $a_i$ is the length of the vertical arrow from $x_i$ to $x_{i+1}$. Here $a_n$ is the last arrow before the point of symmetry.

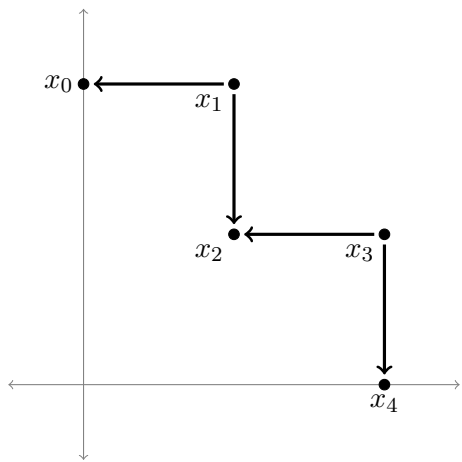
\begin{figure}[htbp]
\label{Ctensor}

\begin{tikzpicture}
	\begin{scope}[thin, gray]
		\draw [<->] (-1, 0) -- (5, 0);
		\draw [<->] (0, -1) -- (0, 5);
	\end{scope}
	
	\filldraw (0, 4) circle (2pt) node[] (y0){};
	\filldraw (2, 4) circle (2pt) node[] (y1){};
	\filldraw (2, 2) circle (2pt) node[] (y2){};
	\filldraw (4, 2) circle (2pt) node[] (y3){};
	\filldraw (4, 0) circle (2pt) node[] (y4){};

	\draw [very thick, <-] (y0) -- (y1);
	\draw [very thick, <-] (y4) -- (y3);	
	\draw [very thick, <-] (y2) -- (y1);
	\draw [very thick, <-] (y2) -- (y3);

	\node [left] at (y0) {$x_0$};
	\node [below left] at (y1) {$x_1$};
	\node [below left] at (y2) {$x_2$};
	\node [below left] at (y3) {$x_3$};
	\node [below] at (y4) {$x_4$};

\end{tikzpicture}
\caption{A staircase complex $(1,1) = CFK^{\infty}(T_{2,5})$. Note that the fact that the lengths of each of the arrows is 1 corresponds to the fact that the Alexander polynomial of $T_{2,5}$ is $1 - t + t^2 -t^3 +t^4$.}
\end{figure}

Note that every such sequence $(a_0, \hdots, a_n)$ defines a staircase complex, but it is not clear that this chain complex must arise as the knot Floer complex of some knot. However, the definition of $\varepsilon$ (and the ordering on $\F$ defined in the next subsection) may be extended formally to include all staircase complexes. Therefore we will often speak of $\varepsilon$-equivalence of staircases and their tensor products. For a more detailed explanation, see  \cite{HHN12}. 

We will often discuss the knot Floer complex of a knot up to $\varepsilon$-equivalence. We can think of $\varepsilon$-equivalence as an equivalence relation on knots and staircase complexes. For example, if $K$ is a knot, then we might write $CFK^{\infty}(K) \sim_{\varepsilon} (a_0, a_1, \hdots, a_n)$ to mean that $CFK^{\infty}(K)$ is $\varepsilon$-equivalent to a staircase complex of the form $(a_0, a_1, \hdots, a_n)$. 

We will often omit all $a_i$ past a certain point for notational convenience; for example, the staircase
\begin{equation*}
(a_0,a_1,a_2, \hdots ,a_n)
\end{equation*}
could be written as 
\begin{equation*}
(a_0, a_1, \hdots ).
\end{equation*}

When describing tensor products of staircase complexes, we will use the symbol $+$ to mean $\otimes$. Repeated tensor products of a complex with itself will be described using multiplicative coefficients in the obvious way. 

The following lemma due to Hancock, Hom, and Newman in \cite{HHN12} will allow us to describe (up to $\varepsilon$-equivalence) the tensor products of certain staircase complexes as staircase complexes:
\begin{lemma}[{\cite[Lemma 3.1]{HHN12}}]
\label{staircaselemma}
Let $a_i, b_j > 0$. If  $m$ is even and $\max \{ a_i \mid i \text{ odd} \} \leq b_j \leq \min \{ a_i \mid i \text{ even} \}$ for $i=1,\hdots ,m$ and $j=1,\hdots ,n$, then
\begin{equation*}
(a_1,a_2,\hdots ,a_m) + (b_1,b_2,\hdots ,b_n) \sim_{\varepsilon} (a_1,a_2,\hdots ,a_m, b_1, b_2, \hdots, b_n).
\end{equation*}

\end{lemma}

\subsection{Ordering of $\F$}

We will begin by recalling some basic facts about totally ordered abelian groups. Let $G$ be a totally ordered abelian group. The absolute value $|a|$ of an element $a$ in $G$ is $a$ if $a \geq id_G$ and $-a$ otherwise. Two elements $a, b \in G$ are \emph{Archimedean equivalent}, denoted $a \sim_{\text{Ar}} b$, if there exist $n,m \in \N$ such that 
\begin{equation*}
|a| \leq n |b| \text{ and } |b| \leq m |a|.
\end{equation*} 
The set of Archimedean equivalence classes of $G$ inherits an ordering from the group. 

If $|a| < |b|$ and $a$ and $b$ are in different Archimedean equivalence classes, we write
\begin{equation*}
|a| \ll |b| 
\end{equation*}
and say that $b$ \emph{dominates} $a$.
The following lemma will be important to later calculations.
\begin{lemma}
\label{archlemma}
Let G be a totally ordered Abelian group, and let $a,b \in G$ such that $a \gg b > 0$. Then
\begin{equation*}
a \pm b \sim_{\text{Ar}} a .
\end{equation*}
\end{lemma}
\begin{proof}
First we note that $(a-b) -a = -b < 0 \Rightarrow  a - b < a$. Now $\forall n \in \mathbb{N}, a- nb > 0 $ since $a \gg b$. Thus we have $a - 2(a-b) = 2b - a = -(a - 2b) < 0 \Rightarrow 2(a-b) > a$. Similarly, $a+b > a$ since $b > 0$, but $2a > a+b$ since $a \gg b$. 
\end{proof}

The invariant $\varepsilon$ induces a total ordering on $\F$ given by 
\begin{equation*}
[CFK^{\infty}(K_1)] > [CFK^{\infty}(K_2)] \text{ if } \varepsilon(K_1 \# - K_2) = 1,
\end{equation*}
where $[CFK^{\infty}(K)]$ denotes the $\varepsilon$-equivalence class of $K$. For notational simplicity, we will often omit the brackets and write $CFK^{\infty}(K_1) > CFK^{\infty}(K_2)$. The following lemmas, due to Hancock, Hom, and Newman in \cite{HHN12} concern the Archimedean equivalence classes of staircase complexes.

\begin{lemma}[{\cite[Lemma 4.1]{HHN12}}]
\label{Farch1}
Let $a_i, b_j > 0$. If $b_1 > a_1$ or if $b_1 = a_1$ and $b_2 < a_2$, then
\begin{equation*}
(a_1,a_2, \hdots ,a_m) \gg (b_1, b_2, \hdots , b_n).
\end{equation*}
\end{lemma}

\begin{lemma}[{\cite[Lemma 4.2]{HHN12}}]
\label{Farch2}
Let $a, c > 0$ and $d \geq 0$. If  $d < c$, then 

\begin{equation*}
(1,a,1,a+c) \gg (1, a, 1, a+d) .
\end{equation*}
\end{lemma}

The proof of the main theorem will also rely on the application of Lemma \ref{Farch2} to staircases of the form 
\begin{equation*}
(1,u,1,v,1,w, \hdots)
\end{equation*}
with  $v \geq u > w \geq 1$. The following lemma justifies this use of Lemma \ref{Farch2}. For the sake of proceeding to the main result, the proof of this lemma will be provided in the fourth section of this paper.

\begin{lemma}
\label{Farch3}
Let $v \geq u > w \geq 1$ be integers. Then the following staircases are Archimedean equivalent
\begin{equation*}
(1,u,1,v) \sim_{\text{Ar}} (1,u,1,v,1,w, \hdots )
\end{equation*}
\end{lemma}

%-----------------------------------------------------------
%
%
%					RESULT
%
%
%-----------------------------------------------------------

\section{Result}
\subsection{Statement of result}
If there exists an order preserving isomorphism between two totally ordered sets $S_1$ and $S_2$ and its inverse is also order preserving, we say $S_1$ and $S_2$ are \emph{order isomorphic}. Order equivalence classes are called \emph{order types}, and the order type of the set of Archimedean equivalence classes of a totally ordered group is called its \emph{coarse order type}. 

A totally ordered group $G$ has a natural filtration indexed by the coarse order type of $G$, where the filtration level corresponding to a given Archimedean equivalence class consists of all of the elements of $G$ with non-strictly smaller Archimedean equivalence class. 

The filtration on $\F$ obtained from its coarse order type can be pulled back to a filtration on $\C$ called the \emph{knot Floer filtration} of $\C$. $\C_{TS}$ inherits this filtration. 

Now we recall the statement of the main theorem first given in the introduction:

\begin{Theorem*}
The indexing set of the knot Floer filtration of $\C_{TS}$ contains a subset that is order isomorphic to $\N \times \N$. We can index the filtration by 
\begin{equation*}
S = \{(i,j) \mid i \in \mathbb{N}, j \in \mathbb{N} \}
\end{equation*}
where S has the lexicographical ordering. Furthermore, each successive quotient is infinite, i.e. for $(i, j), (i',j') \in S$ and $(i, j) < (i',j')$, we have
\begin{equation*}
\mathbb{Z} \subset \F_{(i',j')}/ \F_{(i, j)}. 
\end{equation*}
\end{Theorem*}

The proof of Theorem 1 consists of constructing a family of knots $\T_{n,p}$ for $n,p \in \mathbb{N}$, where each $\T_{n,p}$ generates $\mathbb{Z}$ in $\F_{(n,p)}/ \F_{(i, j)}$ for all $(i,j) < (n,p)$. The construction of $\T_{n,p}$ begins with a different family $\K_{n,p}$ arising from iterated cables of $D$, the Whitehead double of $T_{2,3}$. $\T_{n,p}$ will be constructed from $\K_{n,p}$ by taking successive connected sums with a third family $\S_p$ and certain torus knots and cables of torus knots. The next subsection will define $\K_{n,p}$ and $\S_p$ and will state the staircase structure of their knot Floer complexes. It will also state the knot Floer complexes of the torus knots and cables of torus knots that will be necessary for the proof of the main theorem. Justification of the propositions of the second subsection will be deferred until Section 4. The third subsection will describe the construction of $\T_{n,p}$ and show that these knots satisfy Theorem \ref{mainthm}. 

%----------------------------------------------------------
%
%				Staircase Complexes
%
%----------------------------------------------------------

\subsection{Necessary staircase complexes}

%-----------------------------------------------------------
%				First Lemma
%-----------------------------------------------------------

The first family of knots we will describe is $\K_{n,p}$. Each $\T_{n,p}$ will be constructed from a certain $\K_{n',p'}$ by taking successive connected sums with knots from the families described later in this subsection. Proofs of all statements in this subsection will be postponed until Section 4. 

\begin{definition}
Let $n \geq 4$ be an even integer and $p \geq 5$ an integer. Then the knot $\K_{n,p}$ is defined as
\begin{equation*}
\K_{n,p} = D_{p,p+1;n,n(p^2+p)+1} .
\end{equation*}
\end{definition}

\begin{proposition}
\label{k-staircase}
Let $n \geq 4$ be an even integer and $p \geq 5$ an integer. Then the knot Floer complex $CFK^{\infty}(\K_{n,p})$ is (modulo $\varepsilon$-equivalence) the following:

\begin{equation*}
CFK^{\infty}(\K_{n,p}) \sim_{\varepsilon} (1, n(p+1) -1 ) + (1, n(p-1)-1) + (1, 2n-1, 1, n(p-2) - 1, 1, n-1, \hdots) .
\end{equation*}
\end{proposition}

%----------------------------------------------------------
%					Second Lemma
%----------------------------------------------------------

The next family will also arise as a cable of the Whitehead double of the trefoil knot.

\begin{definition}
Let $q \geq 4$ be an integer. Then the knot $\S_q$ is defined to be
\begin{equation*}
\S_q = \begin{cases}
D_{\frac{q}{2}+1, q+1} & \text{q even} \\
D_{\frac{q+1}{2},q+2} & \text{q odd} .
\end{cases}
\end{equation*}
\end{definition}

\begin{proposition}
\label{sstair}
For any integer $q \geq 4$, the knot Floer complex of $\S_q$ is
\begin{equation*}
CFK^{\infty}(\S_q) \sim_{\varepsilon} (1, q) + (2, \hdots ) .
\end{equation*}
\end{proposition}

%-----------------------------------------------------------
%				Third Lemma
%-----------------------------------------------------------
The two following propositions describe families of torus knots that contain knots that are topologically concordant (but not smoothly concordant) to each member of the family $\S_p$. 

\begin{proposition}
\label{np+1stair}
Up to $\varepsilon$-equivalence, the knot Floer complex of the $(p,n p+1)$ torus knot is 
\begin{equation*}
CFK^{\infty}(T_{p,n p+1}) \sim_{\varepsilon} CFK^{\infty}(T_{p,p+1}) ^{\otimes n} ,
\end{equation*}
where $CFK^{\infty}(T_{p,p+1})$ is given by
\begin{equation*}
CFK^{\infty}(T_{p,p+1}) \sim_{\varepsilon} (1,p-1) + (2, p-2, \hdots) .
\end{equation*}
\end{proposition}

%-----------------------------------------------------------
%						Fourth Lemma
%-----------------------------------------------------------

\begin{proposition}
\label{2p-1stair}
Let $p \geq 2$ be an integer. The knot Floer homology complex of the $(p, 2p-1)$ torus knot is
\begin{equation*}
CFK^{\infty}(T_{p,2p-1}) \sim_{\varepsilon} (1,p-1) + (1,p-2) + (2, \hdots ) .
\end{equation*}
\end{proposition}

%-----------------------------------------------------------
%					Fifth Lemma
%-----------------------------------------------------------
We now describe the knot Floer complex of one final family of knots before moving on to the proof of the main theorem. This family of knots contains knots that are topologically concordant (but not smoothly concordant) to each member of $\K_{n,p}$. 

\begin{proposition}
\label{bigcableprop}
The knot Floer homology complex of the $(n,n(p^2+p)+1)$ cable of the $(p,p+1)$ torus knot is
\begin{equation*}
CFK^{\infty}(T_{p,p+1;n,n(p^2+p)+1}) \sim_{\varepsilon} (1, n p -1) + (1, n-1, \hdots ) . 
\end{equation*}
\end{proposition}

%-----------------------------------------------------------
%
%				Proof of Main Theorem
%
%-----------------------------------------------------------

\subsection{Construction of $\T_{n,p}$} In this subsection, we will construct the family $\T_{n,p}$ from $\K_{n,p}$ and show that $\T_{n,p}$ satisfy the conditions of the main theorem.

Let $n, p \in \N$. Define $n' = 2(n+2)$, $p' = p+5$ and let \begin{equation*}
\T_{n,p}^0 = \K_{n',p'} \# -T_{p', p'+1; n',n'(p'^2+p')+1} .
\end{equation*}
$\T_{n,p}^0$ is topologically slice because $\K_{n',p'}$ is topologically concordant to the given torus knot (since $D$ is topologically concordant to the unknot $U$, so $D_{p',p'+1}$ is topologically concordant to $U_{p',p'+1} = T_{p',p'+1}$, hence their $(n',n'(p'^2+p')+1)$ cables are topologically concordant). The knot Floer complex of $\T_{n,p}^0$ is, by the preceding subsection, $\varepsilon$-equivalent to
\begin{equation*}
(1, n'(p'+1) - 1) + (1, n'(p'-1) - 1)  - (1, n'p'-1)  + (1, 2n' - 1, 1, n'(p'-2) -1, \hdots )- (1, n'-1, \hdots) .
\end{equation*}

Notationally, let $A = (1, 2n' - 1, 1, n'(p'-2) -1, \hdots)$ and $B^0 = -(1, n'-1, \hdots  )$. The goal is to create a knot that is topologically slice and Archimedean equivalent to $A$. By Lemma \ref{archlemma}, only the three terms
\begin{equation*}
(1, n'(p'+1) - 1) \text{, } (1, n'(p'-1) - 1) \text{, and } -(1, n'p'-1)
\end{equation*}
in the knot Floer complex of $\T_{n,p}^0$ obstruct $\T_{n,p}^0$ from Archimedean equivalence to $A$ since by Lemma \ref{Farch1}, $A \gg B^0$. We will remove these obstructions by creating a sequence of topologically slice knots $\T_{n,p}^i$ such that the Archimedean equivalence classes of the obstructions are strictly smaller at each $i$ than in all previous $i$. More precisely, for each $i \geq 1$,  $\T_{n,p}^i$ will have a knot Floer complex of the form 
\begin{equation}
\label{desired_form}
\sum_j (1, a_{i,j}) + A + B^i
\end{equation}
where $A \gg B^i$ and each $a_{i,j}$ satisfies $a_{i-1,j'} > a_{i,j} \geq 2 n'-1$ for some $a_{i-1,j'}$. Now let
\begin{align*}
\T_{n,p}^1 & = \T_{n,p}^0 \# -\S_{n'(p'+1) -1} \# - \S_{n'(p'-1)-1} \# \S_{n'p'-1} \\
& \# T_{(n+2)(p+6),2(n+2)(p+6) +1} \# T_{(n+2)(p+4),2(n+2)(p+4)+1} \# - T_{(n+2)(p+5), 2(n+2)(p+5)+1}.
\end{align*}
This is topologically slice since $\T_{n,p}^0$ is topologically slice and each member of the family $\S$ has an offsetting torus knot of the same topological concordance class. Its knot Floer complex is, by the previous subsection,
\begin{equation*}
CFK^{\infty}(\T_{n,p}^1) \sim_{\varepsilon} 2 (1, (n+2)(p+6) -1) + 2 (1,(n+2)(p+4)-1) - 2(1, (n+2)(p+5)-1) + A + B^1
\end{equation*}
where
\begin{equation*}
B^1 = B^0 + 2 (2, (n+2)(p+6) -2) + 2 (2,(n+2)(p+4)-2) - 2(2, (n+2)(p+5)-2)+ (2, \hdots ) + (2, \hdots) - (2, \hdots ).
\end{equation*}

Each $\T_{n,p}^{i+1}$ is constructed from $\T_{n,p}^{i}$ similarly.  Every remaining staircase summand in the knot Floer complex of $\T_{n,p}^{i}$ that is not part of $A$ or $B^i$ is eliminated by adding an appropriate $\S_q$ or $-\S_q$. So that the resulting knot remains topologically slice, for each of these $\S_q$ a torus knot that is topologically concordant to $\S_q$ is also added with opposite sign. All of the resulting summands $(a_0, a_1, \hdots )$ with $(a_0, a_1, \hdots ) \ll A$ are added to $B^{i}$ to form $B^{i+1}$. 

By Propositions \ref{sstair}, \ref{np+1stair}, and \ref{2p-1stair}, for each $\pm (1, q)$ summand in $CFK^{\infty}(\T_{n,p}^i)$, this process results in a sum of staircases of the form 
\begin{equation*}
\mp (2,\hdots) \pm \left( (1,q_1) + (2,\hdots) \right) \pm  \left((1,q_2) + (2,\hdots) \right) 
\end{equation*}
being added to the complex of $\T_{n,p}^{i+1}$, where $q_1, q_2 < q$. The first term is the difference between $(1,q)$ and $\S_q$ and the remaining terms come from the offsetting torus knot. Since $(2, \hdots) \ll A$, these terms are absorbed into $B^{i+1}$. If $q_1$ or $q_2 < 2 n'-1$, then by by Lemma \ref{Farch1}, the corresponding staircase is dominated by $A$, so it is absorbed into $B^{i+1}$ as well. Then each $CFK^{\infty}(\T_{n,p}^{i})$ for $i \geq 1$ is of the form described in (\ref{desired_form}). Hence there exists some integer $m$ such that 

\begin{equation*}
CFK^{\infty}(\T_{n,p}^m) \sim_{\varepsilon} A + B^m .
\end{equation*}

\begin{definition}
Let $n, p \in \N$, and let $m'$ be the smallest such $m$ as above for the family $\T_{n,p}^{i}$. Then we define $\T_{n,p}$ as
\begin{equation*}
\T_{n,p} = \T_{n,p}^{m'} .
\end{equation*}
\end{definition}

\begin{proof}[Proof of Theorem \ref{mainthm}]
By Lemmas \ref{archlemma} and \ref{Farch1},
\begin{equation*}
CFK^{\infty}(\T_{n,p}) \sim_{\text{Ar}} (1, 2n'-1, 1, n'(p'-2)-1, 1,n'-1,\hdots) ,
\end{equation*}
where $n', p'$ are defined as in the construction above. So by Lemma \ref{Farch1} 
\begin{equation*}
\T_{0,0} \ll \T_{1,0} \ll \hdots \ll \T_{n,0} \ll \hdots ,
\end{equation*}
and since by Lemma \ref{Farch3}
\begin{equation*}
(1, 2n'-1, 1, n'(p'-2) -1, 1, n'-1, \hdots) \sim_{\text{Ar}} (1, 2n'-1, 1, n'(p'-2) ) ,
\end{equation*}
then by Lemma \ref{Farch2}, for any $n, p \in \N$,
\begin{equation*}
\T_{n,0} \ll \T_{n, 1} \ll \hdots \ll \T_{n,p} \ll \hdots  \ll T_{n+1,0} ,
\end{equation*}
establishing the main theorem. 
\end{proof}

\section{Calculations and proof of propositions}

The proof of the main theorem relied on several propositions about the structure of the knot Floer complexes of certain families of iterated cable knots and their connected sums. The purpose of this section is to provide proof of these propositions.

\subsection{Lemma \ref{Farch3}}
The first statement requiring proof was Lemma \ref{Farch3}, which established $\varepsilon$-equivalence of staircases of the form
\begin{equation*}
(1,u,1,v)
\end{equation*}
and those of the form
\begin{equation*}
(1,u,1,v,1,w, \hdots)
\end{equation*}
for $v \geq u > w \geq 1$.

\begin{proof}[Proof of Lemma \ref{Farch3}]
First, we will show that
\begin{equation*}
(1,u,1,v,1,w, \hdots ) > (1,u,1,v)
\end{equation*}
i.e. that
\begin{equation*}
\varepsilon \left( (1,u,1,v,1,w, \hdots ) + (1,u,1,v)^{*} \right) = 1.
\end{equation*}

Let $A = (1,u,1,v,1,w, \hdots )$ and $B= (1,u,1,v)$. Label the first 7 generators of $A$
\begin{equation*}
a,b,c,d,e,f, \text{ and } g ,
\end{equation*}
i.e. $b$ is the generator with $\partial b = a + c$ where the arrow from $b$ to $a$ points to the left and has length 1 and the arrow from $b$ to $c$ points down and has length $u$, $d$ is the generator with $\partial d = c+ e$, etc. 

Label the last six generators of $B^*$
\begin{equation*}
x_6 \text{, } x_5\text{, } x_4\text{, } x_3\text{, } x_2\text{, and } x_1,
\end{equation*}
where $x_1$ is the last generator, $x_2$ is the second-to-last generator, etc. Note that because of symmetry of staircase complexes, the length of the horizontal arrows from to $x_1$ to $x_2$ and from $x_3$ to $x_4$ are 1 and the length of the vertical arrows from $x_3$ to $x_2$ and $x_5$ to $x_4$ are $u$ and $v$ respectively. 

Consider the tensor product 
\begin{equation*}
A \otimes B^* .
\end{equation*}
Denote by $a_i$ the element $a \otimes x_i$, by $b_i$ the element $b \otimes x_i$, etc. By the basis definition of $\varepsilon$, it suffices to find a basis for $A \otimes B^*$ with an element $u_0$ such that $u_0$ has a single incoming horizontal arrow and no outgoing horizontal arrows, as well as no incoming or outgoing vertical arrows.

Let $(x,y)$ denote the position in the lattice of $a_1$. Note that $b_2$, $c_3$, $d_4$, and $e_5$ are also at position $(x,y)$. The boundary maps of these elements are given by
\begin{equation*}
\partial e_5 = e_4 + e_6 \text{, }
\partial d_4 = e_4 + c_4 \text{, }
\partial c_3 = c_4 + c_2 \text{, }
\partial b_2 = c_2 + a_2 \text{, }
\partial a_1 = a_2 .
\end{equation*}
The arrow from $e_5$ to $e_6$ is horizontal of length $v$, so the generator $f_6$ with boundary $\partial f_6 = e_6 + g_6$ is at position $(x-v+1,y)$. Consider the basis change
\begin{itemize}
\item $a_1' = a_1 + b_2 + c_3 + d_4 + e_5 + f_6$,
\item $\alpha_i' = \alpha_i \text{ for all other basis elements } \alpha_i$.
\end{itemize}

Since each of the elements being added to $a_1$ are contained in $ \left( A \otimes B^* \right) _{x,y}$, this is a filtered change of basis. $\partial a_1' = g_6$, and $g_6$ is at the lattice point $(x-v+1, y-w)$, so since $v \geq 2$, the arrow from $a_1'$ to $g_6$ is diagonal (neither vertical nor horizontal). $a_1'$ is not in the image of any vertical differential because $a_1$ is not. $\partial b_1' = a_1' + c_3' + d_4' + e_5' + f_6'$, so there is an arrow from $b_1'$ to $a_1'$. Since $a_1$ is not in image of any basis element other that $b_1$, the arrow from $b_1'$ to $a_1'$ is the unique horizontal arrow into $a_1'$, so $\varepsilon(A \otimes B^*) = 1$. 

Next, we will complete the proof by showing that
\begin{equation*}
2(1,u,1,v) > (1,u,1,v,1,w, \hdots).
\end{equation*}
Lemma 3.2 of \cite{HHN12} establishes that
\begin{equation*}
2(1, u, 1, v) \sim_{\varepsilon} (1, u, 1, v, 1, u, 1, v).
\end{equation*}
Let $A$ be as above, and $C = (1,u,1,v,1,u,1,v)$. It suffices to show that
\begin{equation*}
\varepsilon(C \otimes A^*) = 1.
\end{equation*}
The proof of this fact will follow similarly to the proof that $\varepsilon(A \otimes B^*) = 1$ above.

Denote the first 8 generators of $C$ by
\begin{equation*}
a, b, c, d, e, f, g, \text{ and } h,
\end{equation*}
and denote the last 8 generators of $B^*$ by 
\begin{equation*}
x_8, x_7, \hdots , x_1 .
\end{equation*}
As above, a letter $\alpha \in \{a, b, \hdots h \}$ with subscript $i \in \{1, \hdots, 8 \}$ will denote the tensor product $\alpha \otimes x_i$. 

If $a_1$ is at lattice position $(x,y)$, then so are $b_2$, $c_3$, $d_4$, $e_5$, and $f_6$. Their images under the differential are given by 
\begin{equation*}
\partial a_1 = a_2\text{, } \partial b_2 = a_2 + c_2 \text{, } \partial c_3 = c_2 + c_4 \text{, } \partial d_4 = c_4 + e_4 \text{, } \partial e_5 = e_4 + e_6 \text{, } \partial f_6 = e_6 + g_6 .
\end{equation*}
We also have the generator $g_7$ at lattice position $(x, y - u+w)$ with image under the differential $\partial g_7 = g_6 + g_8$. Consider the basis change 
\begin{itemize}
\item $a_1' = a_1 + b_2 + c_3 + d_4 + e_5 + f_6 + g_7$,
\item $\alpha_i' = \alpha_i$ for all other generators $\alpha_i$.
\end{itemize}
This is a filtered change of basis, and $\partial a_1' = g_8$, so since $u > w$, the image of $a_1'$ is represented by a single diagonal arrow. $\partial b_1' = a_1' + c_3' + d_4' + e_5' + f_6' + g_7'$, so there is a horizontal arrow from $b_1'$ to $a_1'$, but there are no arrows from any other generator to $a_1'$ since $a_1$ is not the image of any basis element other than $b_1$. Hence $\varepsilon(C \otimes B^*) = 1$, completing the proof.
\end{proof}

\subsection{Propositions \ref{k-staircase} and \ref{np+1stair}}
Next, we will establish Proposition \ref{k-staircase}, which stated that the knot Floer complex of $\K_{n,p}$ is $\varepsilon$-equivalent to the staircase complex

\begin{equation*}
(1, n(p+1) -1 ) + (1, n(p-1)-1) + (1, 2n-1, 1, n(p-2) - 1, 1, n-1, \hdots) .
\end{equation*}

We first must prove several lemmas about the knot Floer complexes and Alexander polynomials of torus knots and their cables. Along the way we will prove Proposition \ref{np+1stair}. 

\begin{lemma}
\label{np+1poly}
Let $n \geq 1$ and $p \geq 2$ be integers. Then the Alexander polynomial of the $(p, n p+1)$ torus knot is
\begin{equation*}
\Delta_{T_{p,np+1}}(t)=\sum_{i=0}^{n(p-1)}{t^{ip}}-\sum_{j=0}^{p-2}{\sum_{k=0}^{n-1}{t^{kp+j(pn+1)+1}}} .
\end{equation*}
\end{lemma}

\begin{corollary}[Proposition \ref{np+1stair}]
Up to $\varepsilon$-equivalence, the knot Floer complex of the $(p,n p+1)$ torus knot is 
\begin{equation*}
CFK^{\infty}(T_{p,n p+1}) \sim_{\varepsilon} CFK^{\infty}(T_{p,p+1}) ^{\otimes n} ,
\end{equation*}
where $CFK^{\infty}(T_{p,p+1})$ is given by
\begin{equation*}
CFK^{\infty}(T_{p,p+1}) \sim_{\varepsilon} (1,p-1) + (2, p-2, \hdots) .
\end{equation*}
\end{corollary}

\begin{proof}[Proof of Lemma \ref{np+1poly}]
    
   Since
   
\begin{equation}
\label{torus_poly}
\Delta_{T_{p,q}}(t) = \frac{(t-1)(t^{p q} -1)}{(t^p-1)(t^q-1)},
\end{equation}
it suffices to check
    
   \begin{equation}
   \label{np_to_check}
   (t^p-1)\left(\sum_{i=0}^{n(p-1)}{t^{ip}}-\sum_{j=0}^{p-2}{\sum_{k=0}^{n-1}{t^{kp+j(pn+1)+1}}}\right)=(t-1)\sum_{j=0}^{(p-1)} t^{j(np+1)} .
   \end{equation}

We note that the terms of the LHS of (\ref{np_to_check}) telescope, leaving
     \begin{equation}
     \label{np_lhs_1}
    (t^p-1)\sum_{i=0}^{n(p-1)}{t^{ip}}=t^{(n(p-1)+1)p}-1
    \end{equation}
and
    \begin{equation}
    \label{np_lhs_2}
    -(t^p-1)\Big(\sum_{j=0}^{p-2}{\sum_{k=0}^{n-1}{t^{kp+j(pn+1)+1}}}\Big)=-(t^{np}-1)\sum_{j=0}^{p-2}{t^{j(pn+1)+1}} .
    \end{equation}

   In the RHS of (\ref{np_lhs_2}), $t^{n p} \sum_{j=0}^{p-2}{t^{j(pn+1)+1}} = \sum_{j=1}^{p-1} t^{j (np+1)}$, so the equations in (\ref{np_lhs_2}) also equal
    \begin{equation}
    \label{np_lhs_3}
    -\left(t^{(p-1)(np+1)}-t\right)+(t-1)\sum_{j=1}^{p-2}{t^{j(np+1)}} .
    \end{equation}

     Thus by (\ref{np_lhs_1}) and (\ref{np_lhs_3}), the LHS of (\ref{np_to_check}) equals
    \begin{equation*}
t^{(n(p-1)+1)p}-1-\left(t^{(p-1)(np+1)}-t\right)+(t-1)\sum_{j=1}^{p-2}{t^{j(np+1)}} .
    \end{equation*}
    So by factoring out $t-1$, we get:
    \begin{equation}
    (t^p-1)\left(\sum_{i=0}^{n(p-1)}{t^{ip}}-\sum_{j=0}^{p-2}{\sum_{k=0}^{n-1}{t^{kp+j(pn+1)+1}}}\right)=(t-1)\left(t^{(np+1)(p-1)}+1+\sum_{j=1}^{p-2}{t^{j(np+1)}}\right) .
    \end{equation}

    Note that the $(p-1)^{th}$ term of the RHS of (\ref{np_to_check}) equals $t^{(np+1)(p-1)}$ and the $0^{th}$ term equals 1, establishing equation (\ref{np_to_check}).

\end{proof}

\begin{proof}[Proof of Proposition \ref{np+1stair}]
As a positive torus knot, the knot Floer complex $T_{p,np+1}$ is a staircase complex whose stair lengths are determined by the difference between the degrees of successive terms of its Alexander polynomial. We will start with the case $n=1$. 

Examining the degrees of the monomials in the Alexander polynomial of $T_{p,p+1}$, we see that $CFK^{\infty}(T_{p,p+1})$ is given by the staircase complex
\begin{equation*}
(1,p-1,2,p-2,3,p-3,\hdots ,\lfloor p/2 \rfloor, \lceil p/2 \rceil) ,
\end{equation*}
which, by Lemma \ref{staircaselemma}, is $\varepsilon$-equivalent to 
\begin{equation*}
(1,p-1) + (2, p-2, \hdots) .
\end{equation*}

Moving on to $T_{p,n p + 1}$ for $n > 1$, its knot Floer complex is
\begin{equation*}
(1, p-1, \hdots , 1, p-1, 2, p-2, \hdots , 2, p-2, 3, p-3, \hdots) ,
\end{equation*}
where for each $1 \leq r \leq \lfloor p/2 \rfloor$ the ellipsis stands in for $n$ successive pairs of stairs of length $r, p-r$. Then by rearranging and repeated application of Lemma \ref{staircaselemma}, we have
\begin{align*}
CFK^{\infty}(T_{p, np +1}) &\sim_{\varepsilon} (1, p-1, \hdots , 1, p-1) + (2, p-2, \hdots , 2, p-2) + \hdots + (\lfloor p/2 \rfloor, \lceil p/2 \rceil , \hdots , \lfloor p/2 \rfloor, \lceil p/2 \rceil) \\
& \sim_{\varepsilon} n(1,p-1) + n(2,p-2), \hdots , n(\lfloor p/2 \rfloor , \lceil p/2 \rceil) \\
& \sim_{\varepsilon} n[(1,p-1) + (2, p-2), \hdots (\lfloor p/2 \rfloor , \lceil p/2 \rceil)] \\
& \sim_{\varepsilon} n(1,p-1,2,p-2, \hdots) ,
\end{align*}
establishing the proposition.
\end{proof}

\begin{lemma}
\label{pcable}
Let $p,n$ be positive integers with $p \geq 2$. Then the Alexander polynomial of the $(p,p+1)$ cable of the $(2,3)$ torus knot is

\begin{equation*}
\Delta_{T_{2,3;p,p+1}}(t) = \sum_{i=0}^{p} t^{i (p+1)} + \sum_{i=2}^{p-1} t^{i p} - \sum_{i=0}^{p-2} t^{i(p+1) + 1} - \sum_{i=2}^p t^{i(p+1)-1} .
\end{equation*}

\end{lemma}

\begin{proof}
Recall that 
\begin{equation*}
\Delta_{T_{r,s;p,q}} (t) = \Delta_{T_{r,s}}(t^p) \cdot \Delta_{T_{p,q}} (t).
\end{equation*}
Since $\Delta_{T_{2,3}}(t)= 1-t+t^2$, by Lemma \ref{np+1poly}, 
\begin{equation}
\label{pstart}
\Delta_{T_{2,3;p,p+1}}(t) = (1-t^p+t^{2p}) \cdot \left( \sum_{i=0}^{p-1} t^{i p} - \sum_{i=0}^{p-2} t^{i (p+1)+1}\right) .
\end{equation} 
Considering the RHS of (\ref{pstart}), we will first expand $(1-t^p+t^{2p}) \cdot \sum_{i=0}^{p-1} t^{i p}$. The equation
\begin{equation*}
(1-t^p) \cdot \sum_{i=0}^{p-1} t^{i p}
\end{equation*}
telescopes, leaving only $1-t^{p^2}$, and
\begin{equation*}
t^{2p} \cdot \sum_{i=0}^{p-1} t^{i p} = \sum_{i=2}^{p+1} t^{i p},
\end{equation*}
so 
\begin{equation}
\label{phalf1}
(1-t^p+t^{2 p}) \cdot \sum_{i=0}^{p-1} t^{i p}=1-t^{p^2}+\sum_{i=2}^{p+1} t^{i p} = 1 +t^{p^2+p}+ \sum_{i=2}^{p-1} t^{i p} .
\end{equation}
Returning to the remaining part of the RHS of (\ref{pstart}),  
\begin{equation*}
-t^p \left( -\sum_{i=0}^{p-2} t^{i (p+1) + 1} \right) = \sum_{i=1}^{p-1} t^{i (p+1)}
\end{equation*}
and 
\begin{equation*}
t^{2 p } \left( -\sum_{i=0}^{p-2} t^{i (p+1) + 1} \right) = -\sum_{i=2}^{p} t^{i (p+1) -1},
\end{equation*}
so 
\begin{equation}
\label{phalf2}
-(1-t^p + t^{2p}) \cdot \sum_{i=0}^{p-2} t^{i (p+1) + 1} = -\sum_{i=0}^{p-2} t^{i (p+1) + 1} -\sum_{i=2}^{p} t^{i (p+1) -1}+\sum_{i=1}^{p-1} t^{i (p+1)} .
\end{equation}
Combining (\ref{phalf1}) and (\ref{phalf2}) gives (\ref{pstart}) as desired.
\end{proof}

With the preceding lemmas, we are ready to provide proof of Proposition \ref{k-staircase}. 

\begin{proof}[Proof of Proposition \ref{k-staircase}]

$D$ is $\varepsilon$-equivalent to $T_{2,3}$ (due to Hedden in \cite{Hed07}, see also \cite{Hom11}), so $D_{p,p+1} \sim_{\varepsilon} T_{2,3;p,p+1}$ by Proposition 5.1 of \cite{Hom11}. Hence $\K_{n,p} \sim_{\varepsilon} T_{2,3;p,p+1;n,n(p^2+p)+1}$. Thus it suffices to show that the chain complex of $T_{2,3;p,p+1;n,n(p^2+p)+1}$ takes the desired form. 

$T_{2,3}$ is an L-space knot, so since $g(T_{2,3})=1$, $(p+1)/p \geq 2g(T_{2,3}) -1$, so $T_{2,3;p,p+1}$ is an L-space knot. Its genus is $(p^2+p)/2$, so $(n(p^2+p)+1)/n \geq g(T_{2,3;p,p+1}) -1$, hence $T_{2,3;p,p+1;n,n(p^2+p)+1}$ is also an L-space knot. Thus $CFK^{\infty}(K)$ is determined by gaps between the degree of successive terms of the Alexander polynomial of $T_{2,3;p,p+1;n,n(p^p+p)+1}$. The Alexander polynomial is given by
\begin{equation}
\Delta_{T_{2,3;p,p+1;n,n(p^2+p)+1}} (t) = \Delta_{T_{2,3;p,p+1}}(t^n) \Delta_{T_{n,n(p^2+p)+1}}(t) .
\end{equation}

By Lemmas \ref{np+1poly} and \ref{pcable} we have:
\begin{equation}
\label{p1}
\Delta_{T_{2,3;p,p+1}}(t^n) = 
\displaystyle\sum\limits_{i=0}^{p} t^{n i (p+1)}
+ \displaystyle\sum\limits_{i=2}^{p-1} t^{n i p}
- \displaystyle\sum\limits_{i=0}^{p-2} t^{n i (p+1) + n}
- \displaystyle\sum\limits_{i=2}^{p} t^{n i (p+1) - n}
\end{equation}
and
\begin{equation}
\label{p2}
\Delta_{T_{n,n(p^2+p)+1}}(t) = \displaystyle\sum\limits_{i=0}^{(n-1)(p^2+p)}t^{i n}
-\displaystyle\sum\limits_{j=0}^{n-2} \displaystyle\sum\limits_{k=0}^{p^2+p-1} t^{n k + j((p^2+p)n + 1) + 1}  .
\end{equation}

The proof will proceed in three parts. First, we will expand the terms of a part of the product above (the part that corresponds to all of the positive terms in the Alexander polynomial). Using this expansion, we will examine the smallest terms of the Alexander polynomial of $T_{2,3;p,p+1;n,n(p^2+p)+1}$ to conclude that, for $C = CFK^{\infty}(T_{2,3;p,p+1;n,n(p^2+p)+1})$, 
\begin{equation}
\label{ksecondpart}
C=(1,n(p+1)-1,1,n(p-1)-1,1,2n-1,1,n(p-2)-1,1,n-1,\hdots)
 \end{equation}
Then we will use Lemma \ref{staircaselemma} to show that this staircase is $\varepsilon$-equivalent to the one in the statement of Proposition \ref{k-staircase}. 

First, we will expand the part of the Alexander polynomial given by 
\begin{equation}
\label{Delta+}
\displaystyle\sum\limits_{i=0}^{(n-1)(p^2+p)} t^{i n} \left(
\displaystyle\sum\limits_{i=0}^{p} t^{n i (p+1)}
+ \displaystyle\sum\limits_{i=2}^{p-1} t^{n i p}
- \displaystyle\sum\limits_{i=0}^{p-2} t^{n i (p+1) + n}
- \displaystyle\sum\limits_{i=2}^{p} t^{n i (p+1) - n} \right).
\end{equation}
We will start by looking at 
\begin{equation}
\label{a13}
\displaystyle\sum\limits_{i=0}^{(n-1)(p^2+p)} t^{i n}
\left(\displaystyle\sum\limits_{j=0}^{p} t^{n j (p+1)}- \displaystyle\sum\limits_{j=2}^{p} t^{n j (p+1) - n} \right)  ,
\end{equation}
which is equal to
\begin{align*}
&\displaystyle\sum\limits_{i=0}^{(n-1)(p^2+p)} \displaystyle\sum\limits_{j=0}^{p} t^{i n + n j (p+1)}
-\displaystyle\sum\limits_{i=0}^{(n-1)(p^2+p)} \displaystyle\sum\limits_{j=2}^{p} t^{(i-1) n + n j (p+1)}
\\ &=\displaystyle\sum\limits_{i=1}^{(n-1)(p^2+p)+1} \displaystyle\sum\limits_{j=2}^{p} t^{(i-1) n + n i (p+1)}-\displaystyle\sum\limits_{i=0}^{(n-1)(p^2+p)} \displaystyle\sum\limits_{j=2}^{p} t^{(i-1) n + n i (p+1)} + \displaystyle\sum\limits_{i=0}^{(n-1)(p^2+p)} t^{i n} + \displaystyle\sum\limits_{i=0}^{(n-1)(p^2+p)} t^{i n + (p+1)n}.
\end{align*}
The terms in the two double sums above telescope, leaving only the cases $i=0$ and $i=(n-1)(p^2+p)+1$, so (\ref{a13}) is equal to

\begin{equation}
\displaystyle\sum\limits_{j=2}^p t^{(n-1)(p^2+p)n + nj(p+1)} - \displaystyle\sum\limits_{j=2}^{p} t^{j n (p+1) - n}  + \displaystyle\sum\limits_{i=0}^{(n-1)(p^2+p)} t^{i n} + \displaystyle\sum\limits_{i=0}^{(n-1)(p^2+p)} t^{i n + (p+1)n}.
\end{equation}
Now we will add in the remaining terms to get the following expression for the product in (\ref{Delta+}):
\begin{align*}
& \displaystyle\sum\limits_{j=2}^p t^{(n-1)(p^2+p)n + nj(p+1)} - \displaystyle\sum\limits_{j=2}^{p} t^{j n (p+1) - n}  + \displaystyle\sum\limits_{i=0}^{(n-1)(p^2+p)} t^{i n} + \displaystyle\sum\limits_{i=0}^{(n-1)(p^2+p)} t^{i n + (p+1)n} 
\\ &+ \displaystyle\sum\limits_{i=0}^{(n-1)(p^2+p)} \displaystyle\sum\limits_{j=2}^{p-1} t^{i n + j n p} 
-\displaystyle\sum\limits_{i=0}^{(n-1)(p^2+p)} \displaystyle\sum\limits_{j=0}^{p-2} t^{(i+1) n + j n (p+1)}. 
\end{align*}
Pulling out the $j=p-1$ term from the positive double sum and the $j=0$ and $j=1$ terms from the negative double sum and regrouping gives
\begin{align*}
&\displaystyle\sum\limits_{j=2}^p t^{(n-1)(p^2+p)n + nj(p+1)} - \displaystyle\sum\limits_{j=2}^{p} t^{j n (p+1) - n} + \left(\displaystyle\sum\limits_{i=0}^{(n-1)(p^2+p)} t^{i n} - \displaystyle\sum\limits_{i=0}^{(n-1)(p^2+p)} t^{(i+1) n}\right) 
\\ &+ \left(\displaystyle\sum\limits_{i=0}^{(n-1)(p^2+p)} t^{i n + (p+1)n} - \displaystyle\sum\limits_{i=0}^{(n-1)(p^2+p)} t^{(i+1) n + (p+1)n} \right) + \displaystyle\sum\limits_{i=0}^{(n-1)(p^2+p)} t^{i n + n p (p-1)}
\\ & + \left( \displaystyle\sum\limits_{i=0}^{(n-1)(p^2+p)} \displaystyle\sum\limits_{j=2}^{p-2} t^{i n + j n p} 
-\displaystyle\sum\limits_{i=0}^{(n-1)(p^2+p)} \displaystyle\sum\limits_{j=2}^{p-2} t^{(i+1) n + j n (p+1)}  \right).
\end{align*}
The polynomials in the first and second sets of parentheses telescope in an obvious manner, leaving only the highest and lowest terms (with negative and positive coefficients, respectively). 

Considering the polynomials in the last set of parentheses, note that for each $j \in \{2, \hdots ,p-2\}$, we have 
\begin{align*}
\displaystyle\sum\limits_{i=0}^{(n-1)(p^2+p)} t^{n(i + j p)} - \displaystyle\sum\limits_{i=0}^{(n-1)(p^2+p)} t^{n(i + j p + (j + 1))} &= \displaystyle\sum\limits_{i=0}^{(n-1)(p^2+p)} t^{n(i + j p)} -  \displaystyle\sum\limits_{i=j+1}^{(n-1)(p^2+p)+j+1} t^{n(i + j p )}
\\ &= \displaystyle\sum\limits_{i=0}^{j} t^{n(i + j p)} -
\displaystyle\sum\limits_{i=0}^{j} t^{n(i + j p + (n-1)(p^2+p) + 1)}.
\end{align*}
Canceling and combining yields the following expression for 
the product in (\ref{Delta+}):
\begin{align*}
&\displaystyle\sum\limits_{j=2}^p t^{(n-1)(p^2+p)n + nj(p+1)} - \displaystyle\sum\limits_{j=2}^{p} t^{j n (p+1) - n} + 1 - t^{n(n-1)(p^2+p)+1} + t^{(p+1)n} - t^{n(p+1) + n((n-1)(p^2+p)+1)} 
\\ & + \displaystyle\sum\limits_{i=0}^{(n-1)(p^2+p)} t^{i n + n p (p-1)} + \displaystyle\sum\limits_{j=2}^{p-2}\displaystyle\sum\limits_{i=0}^{j} t^{n(i + j p)} - \displaystyle\sum\limits_{j=2}^{p-2}\displaystyle\sum\limits_{i=0}^{j} t^{n(i + j p + (n-1)(p^2+p) + 1)}.
\end{align*}
From here we note that the degree of each power of $t$ in the terms 
\begin{align*}
& \displaystyle\sum\limits_{j=2}^p t^{(n-1)(p^2+p)n + nj(p+1)}\text{, }
- t^{n(n-1)(p^2+p)+1}\text{, }
- t^{n(p+1) + n((n-1)(p^2+p)+1)} \text{, } \\
& \text{and }
-\displaystyle\sum\limits_{j=2}^{p-2}\displaystyle\sum\limits_{i=0}^{j} t^{n(i + j p + (n-1)(p^2+p) + 1)}
\end{align*}
is greater than $n^2 p(p+1) /2 =g(T_{2,3;p,p+1;n,n(p^2+p)+1})$. Hence these terms occur past the symmetry point of the Alexander polynomial and we can safely ignore them.

Next, pull out the terms $i = j-1$ and $i = j$ from the double sum 
\begin{equation*}
\displaystyle\sum\limits_{j=2}^{p-2}\displaystyle\sum\limits_{i=0}^{j} t^{n(i + j p)}
\end{equation*}
and reorder, yielding

\begin{align*}
&\left(1+t^{(p+1)n} + \displaystyle\sum\limits_{j=2}^{p-2} t^{n j (p+1)}\right) +  \left(\displaystyle\sum\limits_{j=2}^{p-2} t^{ n ((p+1)j - 1)}- \displaystyle\sum\limits_{j=2}^{p} t^{ n((p+1)j - 1)} \right) 
+ \displaystyle\sum\limits_{i=0}^{(n-1)(p^2+p)} t^{i n + n p (p-1)} \\& + \displaystyle\sum\limits_{j=2}^{p-2}\displaystyle\sum\limits_{i=0}^{j-2} t^{n(i + j p)}
\\& = \displaystyle\sum\limits_{j=0}^{p-2} t^{n j (p+1)} - t^{n (p-1) (p+1) - n} - t^{n p (p+1) - n} + \displaystyle\sum\limits_{i=0}^{(n-1)(p^2+p)} t^{i n + n p (p-1)} + \displaystyle\sum\limits_{j=2}^{p-2}\displaystyle\sum\limits_{i=0}^{j-2} t^{n(i + j p)},
\\
\end{align*}
which equals
\begin{equation}
\label{plus_equation}
\displaystyle\sum\limits_{i=0}^{p-3} t^{i n + n p (p-1)} + \displaystyle\sum\limits_{i=p-1}^{2p-2} t^{i n + n p (p-1)} + \displaystyle\sum\limits_{i= 2p-2}^{(n-1)(p^2+p)} t^{i n + n p (p-1)} + \displaystyle\sum\limits_{j=2}^{p-2} t^{n j (p+1)} + \displaystyle\sum\limits_{j=2}^{p-2}\displaystyle\sum\limits_{i=0}^{j-2} t^{n(i + j p)} .
\end{equation}
Pulling out the six terms of lowest degree from the above, we have: 

\begin{equation*}
\displaystyle\sum\limits_{i=0}^{p-3} t^{i n + n p (p-1)} + \displaystyle\sum\limits_{i=p-1}^{2p-2} t^{i n + n p (p-1)} + \displaystyle\sum\limits_{i= 2p-2}^{(n-1)(p^2+p)} t^{i n + n p (p-1)} + \displaystyle\sum\limits_{j=3}^{p-2} t^{n j (p+1)} + \displaystyle\sum\limits_{j=4}^{p-2}\displaystyle\sum\limits_{i=0}^{j-2} t^{n(i + j p)} + p_0(t),
\end{equation*}
where we take the convention that if the index in a sum begins above its maximum value, then the sum evaluates to zero, and where $p_0(t)$ corresponds to the five lowest terms, i.e.
\begin{equation}
\label{p_0}
p_0(t) = 1 + t^{n(p+1)} + t^{2 n p}+t^{2n(p+1)} + t^{3 n p}+ t^{n(3 p + 1)} .
\end{equation}

To establish (\ref{ksecondpart}), we will look at the terms of the product
\begin{equation}
\label{negequation}
-\displaystyle\sum\limits_{j=0}^{n-2} \displaystyle\sum\limits_{k=0}^{p^2+p-1} t^{n k + j((p^2+p)n + 1) + 1} \left(
\displaystyle\sum\limits_{i=0}^{p} t^{n i (p+1)}
+ \displaystyle\sum\limits_{i=2}^{p-1} t^{n i p}
- \displaystyle\sum\limits_{i=0}^{p-2} t^{n i (p+1) + n}
- \displaystyle\sum\limits_{i=2}^{p} t^{n i (p+1) - n}
\right)
\end{equation}
that have degree less than or equal to $3 n p+n$. The terms arising from 
\begin{equation*}
-\displaystyle\sum\limits_{j=0}^{n-2} \displaystyle\sum\limits_{k=0}^{p^2+p-1} t^{n k + j((p^2+p)n + 1) + 1} \left(
\displaystyle\sum\limits_{i=0}^{p} t^{n i (p+1)} \right)
\end{equation*}
with sufficiently small degree are
\begin{equation*}
-\sum_{k=0}^{3p} t^{n k + 1} -\sum_{k=0}^{2 p -1} t^{n k + n(p+1)+1} -\sum_{k=0}^{p-2} t^{n k + 2 n (p+1) + 1},
\end{equation*}
those arising from
\begin{equation*}
-\displaystyle\sum\limits_{j=0}^{n-2} \displaystyle\sum\limits_{k=0}^{p^2+p-1} t^{n k + j((p^2+p)n + 1) + 1} \left(
\sum_{i=2}^{p-1} t^{i n p} \right)
\end{equation*}
are
\begin{equation*}
-\sum_{k=0}^{p} t^{2 n p + n k + 1} - t^{3 n p + 1},
\end{equation*}
those from
\begin{equation*}
-\displaystyle\sum\limits_{j=0}^{n-2} \displaystyle\sum\limits_{k=0}^{p^2+p-1} t^{n k + j((p^2+p)n + 1) + 1} \left(
-\sum_{i=0}^{p-2} t^{n i (p+1) + n} \right)
\end{equation*}
are
\begin{equation*}
\sum_{k=1}^{3p} t^{n k + 1} + \sum_{k = 1}^{2p -1} t^{n k + n(p+1) + 1} + \sum_{k = 1}^{p-2} t^{n k + 2n (p+1) + 1} ,
\end{equation*}
and finally, the terms from the remaining product are
\begin{equation*}
\sum_{i=1}^{p} t^{2 n p + i n + 1}.
\end{equation*}
Combining the above, we are left with
\begin{equation}
\label{p0-}
-t - t^{n (p+1) + 1} - t^{2 n (p+1) + 1} - t^{ 2 n p + 1} - t^{3 n p + 1}.
\end{equation}
The sum of (\ref{p0-}) and (\ref{p_0}) consists of the eleven terms of the polynomial of lowest degree, and by examining the difference in the degrees of successive terms, we arrive at (\ref{ksecondpart}).

By Lemma \ref{staircaselemma}, in order to show that
\begin{equation*} (1,n(p+1)-1,1,n(p-1)-1,1,2n-1,1,n(p-2)-1, \hdots ) 
\end{equation*}
is $\varepsilon$-equivalent to
\begin{equation*}
 (1, n(p+1)-1) + (1, n(p-1) -1) + (1, 2n-1, 1, n(p-2) - 1, \hdots ) ,
\end{equation*}
it suffices to show that the difference between the degree of any two consecutive terms of the Alexander polynomial after the difference between the fourth and the fifth terms and up to the point of symmetry is less than or equal to $n(p-1)-1$. 

We will do so in two steps. First, we will show that none of the negative terms in (\ref{negequation}) have the same degree as any of the terms in (\ref{plus_equation}). This tells us that every power of $t$ in (\ref{plus_equation}) must appear in the desired Alexander polynomial with nonzero coefficient (hence coefficient equal to 1 since the knot in question is an L-space knot). Then we will show that for any monomial $t^q$ in (\ref{plus_equation}) (not including $1$, $t^{n(p+1)}$, or $t^{2 n p}$), there must be another monomial $t^{q'}$ in the sum with $q'-q \leq n(p-1)$. Since the coefficients of successive terms in the Alexander polynomial of an L-space knot alternate between positive and negative, this shows that the difference between any two successive terms must be non-strictly less than $n(p-1)-1$ as desired. 

Observe that the degree of every term in (\ref{plus_equation}) is divisible by $n$. The degree of every monomial with negative coefficient in (\ref{negequation}) differs from a multiple of $n$ by $j+1$, so since $1 \leq j+1 \leq n-1$, none of their degrees is divisible by $n$, so none of the terms in (\ref{negequation}) have the same degree as any of the terms in (\ref{plus_equation}) .

Now note that the difference between degrees of successive terms in each of the first three sums in (\ref{plus_equation}) is $n$, and the difference between the highest degree of the first sum and the lowest of the second sum (and similarly between the second and third sums) is $2 n$. Furthermore, the degree of the highest term in the third sum is $(n-1)(p^2+p)n+n p (p-1)$, which is past the point of symmetry of $\Delta_{K_n(p)} (t)$, so we only need to consider terms with degree less than $n p (p-1)$, the degree of the smallest term in these three sums. 

Next, consider the double sum \begin{equation*}
\displaystyle\sum\limits_{j=3}^{p-2}\displaystyle\sum\limits_{i=0}^{j-2} t^{n(i + j p)} .
\end{equation*}
For any given $j \in \{3, \hdots ,p-2\}$, the terms 
\begin{equation*}
t^{n j p}, t^{n j p + n}, t^{n j p + 2 n}, \hdots , t^{n j p + (j-2)n}
\end{equation*} 
differ successively in degree by $n$. Furthermore, $t^{j n p + (j-2)n}$, the largest of the terms arising from $j$, differs in degree from $t^{n p j + n p}$, the smallest of the terms arising from $j+1$, by $n p - n(j-2) = n(p-j+2)$, so since $j \geq 3$ this difference is smaller than $n(p-1)$. 

Finally,  consider the terms from the sum 
\begin{equation*}
\displaystyle\sum\limits_{j=2}^{p-2} t^{n j (p+1)} .
\end{equation*} The smallest of these is $t^{2n(p+1)}$, which differs in degree from the smallest term $t^{3 n p}$ of the double sum by only $3 n p - 2 n p - 2 n = n(p-2)$. All following terms have degree greater than $3 n p$, so their degree must be within $n(p-1)$ of one of the terms listed above. 

Therefore the difference in degree of any two successive terms after the stairs $(1,n(p+1),1,n(p-1)-1,\hdots)$ must have length less than or equal to $n(p-1)-1$. The result follows from Lemma \ref{staircaselemma}.
\end{proof}

\subsection{Propositions \ref{sstair} and \ref{2p-1stair}}
Next, we move on to the proof of Proposition \ref{sstair}, which stated that the knot Floer complex of $\S_q$ is $\varepsilon$-equivalent to the staircase complex
\begin{equation*}
(1, q) + (2, \hdots ).
\end{equation*}

We will again need a lemma about the Alexander polynomial of a family of torus knots, which will lead to the proof of Proposition \ref{2p-1stair}. 
\begin{lemma}
\label{2p-1poly}
For any integer $p \geq 2$, the Alexander polynomial of the $(p, 2p-1)$ torus knot is 
\begin{equation*}
\Delta_{T_{p,2p-1}}(t) = \sum_{i=0}^{p-2} (t^{(2 p -1)i} + t^{(2p-1)i + p}) - \sum_{i=0}^{2(p-2)} t^{i p + 1} .
\end{equation*}
\end{lemma}
\begin{corollary}[Proposition \ref{2p-1stair}]
The knot Floer homology complex of the $(p, 2p-1)$ torus knot is
\begin{equation*}
CFK^{\infty}(T_{p,2p-1}) \sim_{\varepsilon} (1,p-1) + (1,p-2) + (2, \hdots ).
\end{equation*}
\end{corollary}

\begin{proof}[Proof of Lemma \ref{2p-1poly}]
\begin{equation*}
Since
\frac{t^{p (2p -1)} -1}{t^{2p-1}-1} = 1+ t^{2 p -1} + \hdots + t^{(p-1)(2p-1)},
\end{equation*}
by (\ref{torus_poly}), it suffices to show that
\begin{equation}
\label{2p-1start}
(t^p-1) \left( \sum_{i=0}^{p-2} \left( t^{(2 p -1)i} + t^{(2p-1)i + p} \right) - \sum_{i=0}^{2(p-2)} t^{i p + 1} \right) = (t-1) \left( \sum_{i=0}^{p-1} t^{(2p -1) i} \right).
\end{equation}
The terms of the LHS of (\ref{2p-1start}) telescope, leaving
\begin{equation*}
\sum_{i=0}^{p-2} \left( t^{(2p-1)i + 2p} - t^{(2p-1) i } \right) +t-t^{(2p-1) p + 1},
\end{equation*}
which further simplifies to
\begin{equation}
\label{2p-1done}
\sum_{i=0}^{p-2} (t^{(2p-1)i +1} -t^{(2p-1)i}),
\end{equation}
and factoring out $(t-1)$ from (\ref{2p-1done}) establishes (\ref{2p-1start}).
\end{proof}

\begin{proof}[Proof of Proposition \ref{2p-1stair}]
By inspecting the successive degrees of terms in the Alexander polynomial of $T_{p,2p-1}$, we see that
\begin{equation*}
CFK^{\infty}(T_{p,2p-1}) \sim_{\varepsilon} (1,p-1,1,p-2,2,p-2,2,p-3, \hdots),
\end{equation*}
and the result follows from Lemma \ref{staircaselemma}.
\end{proof}

Now we are ready to prove the Proposition \ref{sstair}.
\begin{proof}[Proof of Proposition \ref{sstair}]
The proof will be split into two cases: $p$ even and $p$ odd. 

If $p$ is even, by similar considerations as in the proof of Proposition \ref{k-staircase}, it suffices to show that the knot Floer complex of $T_{2,3;\frac{p}{2}+1,p+1}$ takes the desired form (up to $\varepsilon$-equivalence). 

We will begin by expanding the Alexander polynomial of $T_{2,3;\frac{p}{2}+1,p+1}$. By Lemma \ref{2p-1poly}, we have
\begin{equation}
\Delta_{T_{2,3;\frac{p}{2}+1,p+1}}(t) = (1-t^{\frac{p}{2}+1}+t^{p+2}) \cdot \left( \sum_{i=0}^{\frac{p}{2}-1} \left( t^{(p+1)i} + t^{(p+1)i + p/2 + 1} \right) - \sum_{i=0}^{p-2} t^{i ( p/2 + 1) + 1} \right).
\end{equation}
The term
\begin{equation*}
(1-t^{\frac{p}{2}+1}+ t^{p+2})\cdot \left( \sum_{i=0}^{\frac{p}{2}-1} \left( t^{(p+1)i} + t^{(p+1)i + p/2 + 1} \right) \right)
\end{equation*}
telescopes, leaving
\begin{equation*}
\sum_{i=0}^{\frac{p}{2}-1} t^{(p+1) i} + \sum_{i = 1}^{\frac{p}{2}} t^{(p+1) i +\frac{p } {2} + 2}, 
\end{equation*}
and 
\begin{equation*}
(1-t^{\frac{p}{2}+1}) \cdot \left( - \sum_{i=0}^{p-2} t^{i ( p/2 + 1) + 1} \right) = -t + t^{(p-1) (p/2 + 1) + 1} ,
\end{equation*}
so we have 
\begin{equation}
\label{Sparta} %woah...
\Delta_{T_{2,3;\frac{p}{2}+1,p+1}}(t) = \sum_{i=0}^{\frac{p}{2}-1} t^{(p+1) i} + \sum_{i = 1}^{\frac{p}{2}} t^{(p+1) i +\frac{p } {2} + 2} -t - \sum_{i=2}^{p-2} t^{i(p/2+1) + 1} - t^{p (p/2+1) + 1} .
\end{equation}
Examining the degrees of successive terms, we see that 
\begin{equation*}
CFK^{\infty}(T_{2,3;\frac{p}{2}+1,p+1}) \sim_{\varepsilon} (1, p,2, p/2,1,p/2-2, \hdots ) ,
\end{equation*}
where the even entries after the ellipsis and before the point of symmetry are all less than $p$, so the result follows in this case from Lemma \ref{staircaselemma}.

The case $p$ odd is similar. Expand the formula
\begin{equation}
\Delta_{T_{2,3;\frac{p+1}{2},p+2}}(t) = \left( 1 - t^{\frac{p+1}{2}} + t^{p+2} \right) \cdot \left( \sum_{i=0}^{p-1} t^{i \frac{p+1}{2}} - \sum_{j=0}^{p-2} \left( t^{j (p+2) + 1} + t^{j(p+2) + \frac{p+1}{2} + 1} \right) \right)
\end{equation}
to obtain
\begin{equation}
\Delta_{T_{2,3;\frac{p+1}{2},p+2}}(t) = 1 + \sum_{i = 2}^{p-2} t^{i \frac{p+1}{2}} + t^{\frac{p^2+p}{2}} - \sum_{j=0}^{p-2} t^{j (p+2) +1} - \sum_{j=1}^{p-1} t^{j(p+2) + \frac{p+1}{2} + 1} ,
\end{equation}
showing that $CFK^{\infty} (T_{2,3; \frac{p+1}{2},p+2})$ is a staircase with the form
\begin{equation}
(1,p,2,\frac{p+1}{2}-2,2, \hdots ) ,
\end{equation}
establishing the proposition for case $p$ odd by Lemma \ref{staircaselemma}.
\end{proof}

\subsection{Proposition \ref{bigcableprop}}

Finally, we will establish Proposition \ref{bigcableprop}, completing the justification of all of the propositions used to prove the main theorem. Proposition \ref{bigcableprop} stated that the knot Floer complex of the $(n, n(p^2+p)+1)$ cable of the $(p,p+1)$ torus knot is $\varepsilon$-equivalent to the staircase complex
\begin{equation*}
(1, n p -1) + (1, n-1, \hdots ) . 
\end{equation*}

\begin{proof}[Proof of Proposition \ref{bigcableprop}]
As before, we will consider the knot's Alexander polynomial. 
\begin{align*}
\Delta_{T_{p,p+1;n,n(p^2+p)+1}}(t) &= \Delta_{T_{n,n(p^2+p)+1}}(t) \cdot \Delta_{T_{p,p+1}}(t^n) \\
&= \left( \sum_{i=0}^{(p^2+p) (n-1)} t^{i n} - \sum_{j=0}^{n-2} \sum_{i=0}^{p^2+p-1} t^{i n + j(n(p^2+p)+1)+1} \right) \left( \sum_{k=0}^{p-1} t^{k p n} - \sum_{k=0}^{p-2} t^{k n (p+1) + n} \right) .
\end{align*}
Consider first 
\begin{equation}
\label{firsttimes}
\left( \sum_{i=0}^{(p^2+p) (n-1)} t^{i n} \right) \left( \sum_{k=0}^{p-1} t^{k p n} - \sum_{k=0}^{p-2} t^{k n (p+1) + n} \right) ,
\end{equation}
which equals 
\begin{equation*}
\sum_{k=0}^{p-1} \sum_{i = 0} ^ {(p^2+p)(n-1)} t^{ i n + k n p} - \sum_{k=0}^{p-2} \sum_{i = 0}^{(p^2+p) (n-1)} t^{i n + k n p + (k+1) n } .
\end{equation*}
Observe that for $0 \leq k \leq p-2$, all the terms except the smallest and largest $k$ telescope, leaving the following expression for (\ref{firsttimes}):
\begin{equation}
\label{nstair_start}
\sum_{k=0}^{p-2} \sum_{i=0}^{k} t^{i n + k n p} - \sum_{k=0}^{p-2} \sum_{i = 0}^{k} t^{k n p + (k+1) n + n (p^2+p)(n-1) - i n} + \sum_{i=0} ^{(p^2+p)(n-1)} t^{(p-1) n p + i n} .
\end{equation}
We can write the first sum as follows, where each row is a $k$ value:
\begin{equation*}
\begin{tabular}{llllll}
$t^0$ \\
$+t^{n p}$ & $+ t^{n p + n}$ \\
$+t^{2 n p}$ & + $t^{2 n p + n}$ & + $t^{2 n p + 2n}$ \\
$\vdots$ & $\vdots$ & $\vdots$ &  $\ddots$\\
$+t^{(p-2) n p}$ & $+t^{(p-2) n p+ n}$ &$+ t^{(p-2) n p + 2n}$ &$+ \hdots$ & $+ t^{(p-2) n p + (p-2) n}$& $+ 0$ .
\end{tabular}
\end{equation*}
Let $m = (n-1)(p+1)$. The difference between the third sum and the second can be written as:
\begin{equation*}
\begin{tabular}{llllll}
$t^{(p-1) n p}$ & $+ t^{(p-1) n p+ n}$  & $+ t^{(p-1) n p + 2 n}$ & $+ \hdots$ &  $+ t^{(p-1) n p + (p-2) n}$ & $+ t^{(p-1) n p + (p-1) n} $\\
$\vdots$ & $\vdots$ & $\vdots$ & $\hdots$ & $\vdots$ & $\vdots$ \\ 
$+t^{(m-1) n p}$ & $+ t^{(m-1) n p+ n}$ & $+ t^{(m-1) n p+ 2n} $ & $+ \hdots$ & $ + t^{(m-1) n p + (p-2) n}$ & $+ t^{(m-1) n p + (p-1) n}$ \\
$+t^{m n p}$ & $+ 0$ & $+ t^{m n p + 2 n}$ &$+ \hdots$ & $+ t^{m n p + (p-2) n}$ &$+ t^{m n p + (p-1) n}$ \\
$+ t^{(m+1) n p}$ & $+0$ & $+ 0$ & $+\hdots$ & $+t^{(m+1) n p+ (p-2) n}$ & $+t^{(m+1) n p+ (p-1) n}$\\
$\vdots$ &$\vdots$  & $\vdots$  & $\hdots$  & $\vdots$  & $\vdots$ \\
$+t^{(m+p-3)n p}$ & $+ 0 $ & $+0$ &$+\hdots$ & $+0$ & + $t^{(m+p-3) n p + (p-1) n}$ \\
$+ t^{(m + p -2)  n p}$ .
\end{tabular}
\end{equation*}
Looking at the columns of the previous two tables, we see that
\begin{equation}
\label{nplusterms}
\sum_{i=0}^{n p + n -3} t^{ i n p} + \sum_{j=1}^{p-1} \sum_{i =0}^{(n -1) (p-1) -1} t^{jn(p+1) + i n p}
\end{equation}
is an expression for  (\ref{nstair_start}). 
The rest of the proof of the proposition is similar to the proof of Proposition \ref{k-staircase}. The smallest terms in (\ref{nplusterms}) are $1 + t^{n p} + t^{np + n} $ and the smallest terms in
\begin{equation}
\label{nminusterms}
\left( - \sum_{j=0}^{n-2} \sum_{i=0}^{p^2+p-1} t^{i n + j(n(p^2+p)+1)+1} \right) \left( \sum_{k=0}^{p-1} t^{k p n} - \sum_{k=0}^{p-2} t^{k n (p+1) + n} \right)
\end{equation}
that do not cancel are $-t - t^{np+1}$, so the beginning of the staircase looks like $(1, np-1, 1, p-1, \hdots )$.

The degree of every monomial in (\ref{nplusterms}) is divisible by $n$, but none of the monomials from (\ref{nminusterms}) have degree divisible by $n$. Finally, it is clear that for every monomial (up to the point of symmetry) $t^q$ in (\ref{nplusterms}), there is another monomial $t^{q'}$ with $q'-q \leq n p$ because such a monomial can be found in the first sum alone. The result follows from Lemma \ref{staircaselemma}.
\end{proof}

\bibliographystyle{amsalpha}
\bibliography{mybib}

\providecommand{\bysame}{\leavevmode\hbox to3em{\hrulefill}\thinspace}
\providecommand{\MR}{\relax\ifhmode\unskip\space\fi MR }
\providecommand{\MRhref}[2]{%
  \href{http://www.ams.org/mathscinet-getitem?mr=#1}{#2}
}
\providecommand{\href}[2]{#2}
\begin{thebibliography}{HHN12}

\bibitem[End95]{End95}
Hisaaki Endo, \emph{Linear independence of topologically slice knots in the
  smooth cobordism group}, Topology Appl. \textbf{63} (1995), 257 -- 262.

\bibitem[Hed07]{Hed07}
Matthew Hedden, \emph{Knot {F}loer homology of {W}hitehead doubles}, Geometry
  \& Topology \textbf{11} (2007), 2277--2388.

\bibitem[Hed09]{Hed09}
\bysame, \emph{On knot {F}loer homology and cabling {II}}, International
  Mathematics Research Notices (2009), 2248--2274.

\bibitem[HHN12]{HHN12}
Stephen Hancock, Jennifer Hom, and Michael Newman, \emph{On the knot {F}loer
  filtration of the concordance group}, preprint (2012), arXiv:1210.4193.

\bibitem[Hom11]{Hom11}
Jennifer Hom, \emph{The knot {F}loer complex and the smooth concordance group},
  To appear in Comment. Math. Helv. (2011), arXiv:1111.6635.

\bibitem[Hom12]{Hom12}
\bysame, \emph{Bordered {H}eegaard {F}loer homology and the tau-invariant of
  cable knots}, preprint (2012), arXiv:1202:1463.

\bibitem[LOT08]{LOT08}
Robert Lipshitz, Peter Ozsv\'{a}th, and Dylan Thurston, \emph{Bordered
  {H}eegaard {F}loer homology: Invariance and pairing}, preprint (2008),
  arXiv:0810.00687v4.

\bibitem[Mos71]{Mo}
Louise Moser, \emph{Elementary surgery along torus knots}, Pacific Journal of
  Math \textbf{38} (1971), 734 -- 745.

\bibitem[OS03]{OS03}
Peter Ozsv\'{a}th and Zolt\'{a}n Szab\'{o}, \emph{Knot {F}loer homology and the
  four-ball genus}, Geometry \& Topology \textbf{7} (2003), 615--639.

\bibitem[OS04]{OS04}
\bysame, \emph{Holomorphic disks and knot invariants}, Advances in Mathematics
  \textbf{186} (2004), no.~1, 58--116.

\bibitem[OS05]{OS05}
Peter Osv\'{a}th and Zolt\'{a}n Szab\'{o}, \emph{On knot {F}loer homology and
  lens space surgeries}, Topology \textbf{44} (2005), no.~6, 1281--1300.

\bibitem[OS11]{OS11}
Peter Ozsv\'ath and Zolt\'an Szab\'o, \emph{Knot {F}loer homology and rational
  surgeries}, Algebraic and Geometric Topology \textbf{11} (2011), 1 -- 68.

\bibitem[Ras03]{Ras03}
Jacob Rasmussen, \emph{Floer homology and knot complements}, Ph.D. thesis,
  Harvard University, 2003.

\end{thebibliography}

\end{document}